\documentclass [12pt,twoside,a4paper]{article}
\usepackage{amsfonts}
\usepackage{amsthm}
\usepackage{amsmath}
\usepackage{amstext}
\usepackage{amssymb}
\usepackage{mathrsfs}
\usepackage{amscd}
\usepackage{xypic}
\usepackage{epsf}              
\usepackage{graphicx}          
\usepackage{fancybox}          
\usepackage{color}             
\usepackage{fancyhdr}
\usepackage[hang,footnotesize]{caption2}  
\usepackage{enumerate} 
\def\mathcal{\mathscr}
\newfont{\aaa}{cmb10 at 19pt}
\newfont{\bbb}{cmb10 at 14pt}

\newtheorem{theorem}{Theorem}

\newtheorem{lemma}{Lemma}[]

\newtheorem{conjecture}{Conjecture}[]

\newtheorem*{conjecture*}{Conjecture}

\newtheorem{claim}{Claim}[]

\theoremstyle{definition}

\theoremstyle{remark}

\makeatletter

\newcommand{\Rmnum}[1]{\expandafter\@slowromancap\romannumeral #1@}
\makeatother

\newcommand{\beq}{\begin{equation}}
\newcommand{\eeq}{\end{equation}}
\newcommand{\bey}{\begin{eqnarray}}
\newcommand{\eey}{\end{eqnarray}}
\newcommand{\beyy}{\begin{eqnarray*}}
\newcommand{\eeyy}{\end{eqnarray*}}

\begin{document}

\thispagestyle{empty} \thispagestyle{fancy} {
\fancyhead[RO,LE]{\scriptsize \bf 
} \fancyfoot[CE,CO]{}}
\renewcommand{\headrulewidth}{0pt}


\setcounter{page}{1}
\qquad\\[8mm]

\centerline{\aaa{The confirmation of a conjecture on disjoint cycles in a graph$^{\ast}$}}

\vskip5mm
\centerline{\bf Fuhong Ma,\quad Jin Yan$^{\dagger}$}
\centerline\footnotesize{{$School \ of \ Mathematics, \ Shandong \ University, \ Jinan \ 250100, \ China$}
\vskip13mm

\normalsize\noindent{\bf Abstract}\quad
In this paper, we prove the following conjecture proposed by Gould, Hirohata and Keller [Discrete Math. submitted]: Let $G$ be a graph of sufficiently large order. If $\sigma_t(G) \geq 2kt - t + 1$ for any two integers $k \geq 2$ and $t \geq 5$, then $G$ contains $k$ disjoint cycles.
\footnotetext{\\$^{\ast}$This work is supported by NNSF of China(No.11671232, 11271230, ).\\$^{\dagger}$Corresponding author: Jin Yan, E-mail: yanj@sdu.edu.cn}
\\


\section{Introduction}

In this paper,  we consider only finite undirected graphs without loops and multiple edges. Let $G$ be a  simple graph.
 A set of subgraphs of $G$ is said to be vertex-disjoint if no two of them have any vertex in common.  Denote by $e(G)$ the number of edges of $G$. For a vertex $x \in V(G)$, the neighborhood of $x$ in $G$ is denoted by $N_G(x)$, and $d_G(x) = |N_G(X)|$ is the degree of $x$ in $G$. For a subgraph $H$ of $G$ and a vertex $x \in V(H)$, we denote $N_H(x) = N_G(x) \cap V(H)$ and $d_H(x) = |N_H(x)|$. For a subgraph $H$ and a subset $S$ of $H$, $d_H(S) = \sum\limits_{x \in S}d_H(x)$. The vertex subgraph induced by $S$ is denoted by $G[S]$, and $G - S = G[V(G) - S]$.
For a graph $G$, $|G| = |V(G)|$ is the order of $G$, $\omega(G)$ is the number of components of $G$, $\delta(G)$
is the minimum degree of $G$, and
\begin{center}
 $\sigma_k(G) = $min$\{\sum\limits_{x \in X}d_{G}(x)$ $|$ $X$ is an independent set of $G$ with $|X|=k\}$.
\end{center}
For graphs $G$ and $H$, $G \cup H$ denotes the union of $G$ and $H$. For a graph $G$, $mG$ denotes the union of $m$ copies of $G$. $K_n$ denotes a complete graph of order $n$.

In this paper, we consider degree sum conditions and the existence of vertex-disjoint cycles. For convenience, we write disjoint instead of vertex-disjoint. Finding proper conditions for disjoint cycles is an interesting problem. In 1962 \cite{Erdos}, Erd$\ddot{o}$s and P$\acute{o}$sa found a condition concerning the number of edges to ensure two disjoint cycles by proving that every graph $G$ of order $n \geq 6$, if $e(G) \geq 3n - 6$, then G has 2
disjoint cycles or is isomorphic to $K_3 + (n - 3)K_1$.  In 1963 \cite{Dirac}, Dirac gave a minimum degree condition for $k$ disjoint triangles. They proved that for $k \geq 1$, any graph G with order at least $n \geq 3k$ and $\delta(G) \geq (n + k) / 2$ contains $k$ disjoint triangles. For general cycles, Corr$\acute{a}$di and Hajnal proved a classical result.

\begin{theorem}\label{C and H} (Corr$\acute{a}$di and Hajnal \cite{candh}) Suppose that $|G| \geq 3k$ and $\delta(G) \geq 2k$. Then $G$ contains $k$ disjoint cycles.
\end{theorem}

Justesen inproved Theorem \ref{C and H} as follows.

\begin{theorem}\label{JUS}(Justesen \cite{just}) Suppose that $|G| \geq 3k$ and $\sigma_2(G) \geq 4k$.  Then $G$ contains $k$ disjoint cycles.
\end{theorem}

The degree condition in Theorem \ref{JUS} is not sharp. Later, Enomoto and Wang independently improved Theorem \ref{JUS} and got a sharp degree bound.

\begin{theorem}\label{E and W}(Enomoto \cite{enomoto}, Wang \cite{wang}) Suppose that $|G| \geq 3k$ and $\sigma_2(G) \geq 4k - 1$. Then $G$ contains $k$ disjoint cycles.
\end{theorem}

In 2006, Fujita, Matsumura, Tsugaki and Yamashita \cite{fujita} gave a  sharp degree sum condition on three independent vertices by prove the following theorem.

\begin{theorem}\label{Fujita}(Fujita et al. \cite{fujita}) Suppose that $k \geq 2$, $|G| \geq 3k + 2$. If $\sigma_3(G) \geq 6k - 2$, then $G$ contains $k$ disjoint cycles.
\end{theorem}

Recently, Gould, Hirohata and Keller proposed a  more general conjecture.

\begin{conjecture}\label{Gould}(\cite{conj}) Let $G$ be a graph of sufficiently large order. If $\sigma_t(G) \geq 2kt - t + 1$ for any two integers $k \geq 2$ and $t \geq 4$, then $G$ contains $k$ disjoint cycles.
\end{conjecture}

They showed that the degree sum condition conjectured above is sharp. Sharpness is given by $G = K_{2k-1}+mK_1$. The only independent vertices in $G$ are those in $mK_1$. Each of these vertices has degree $2k - 1$. Thus $\sigma_t(G) = t(2k - 1) = 2kt - t$ for any $4 \leq t \leq m$. Apparently, $G$ does not contain $k$ disjoint cycles as any cycle must contain two vertices of $K_{2k - 1}$. In the same paper, they also verified that the case $t = 4$ is correct, which adds evidence for this conjecture.

In this paper, we solve Conjecture \ref{Gould} for $t \geq 5$, by proving the following theorem.

\begin{theorem}\label{M and Y} Suppose that $k \geq 2$, $t \geq 5$ are two integers and $|G| \geq (2t - 1)k$. If $\sigma_t(G) \geq 2kt - t + 1$, then $G$ contains $k$ disjoint cycles.
\end{theorem}

Other related results about disjoint cycles in graphs and bipartite graphs have been studied, we refer the reader see \cite{Jiao}, \cite{wang2}, \cite{wang1} and \cite{Yan}.

\medskip

{\bf Remark.} In the following, we introduce some useful notations. Let $X, Y$ be two vertex-disjoint subsets or subgraphs of $G$, $E(X, Y)$ denote the set of edges of $G$ joining a vertex in $X$ and a vertex in $Y$. If $X = \{x\}$, we denote $E(x, Y)$ instead of $E(\{x\}, Y)$. And denote $e(X, Y) = |E(X, Y)|$, $e(x, Y) = |E(x, Y)|$. For two disjoint subgraphs $H_1, H_2$ of $G$, $(d_1, \ldots , d_n)$ (where $d_1 \geq \cdots \geq d_n$) is a degree sequence from $H_1$ to $H_2$ if there exist $n$ vertices $v_1, \ldots , v_n$ in $H_1$ such that $e(v_i, H_2) \geq d_i$ for each $1 \leq i \leq n$.

A forest is a graph each of whose components is a tree. A leaf is a vertex of a forest whose degree is at most 1.

\section{Lemmas}

To prove Theorem \ref{M and Y}, we make use of the following lemmas.

Let $C_1, \ldots , C_k$ be $k$ disjoint cycles of a graph $G$. If $C'_1, \ldots , C'_k$ are $k$ disjoint cycles of $G$ and $|\cup^k_{i = 1}V(C'_i)| < |\cup^k_{i = 1}V(C_i)|$, then we call $\{C'_1, \ldots , C'_k\}$ shorter cycles than $\{C_1, \ldots , C_k\}$. We also call $\{C_1, \ldots , C_k\}$  minimal if $G$ does not contain $k$ disjoint cycles $C'_1, \ldots , C'_k$ such that $|\cup^k_{i = 1}V(C'_i)| < \cup^k_{i = 1}V(C_i)|$.

\begin{lemma}\label{lma} (\cite{fujita}) Let $k$ be a positive integer and $C_1, \ldots , C_k$ be $k$ disjoint cycles of a graph $G$. If $\{C_1, \ldots , C_k\}$ is minimal, then $e(x, C_i) \leq 3$ for any $x \in V(G) - \cup^k_{i = 1}V(C_i)$ and for any $1 \leq i \leq k$. Furthermore, $e(x, C_i) = 3$ implies $|C_i| = 3$ and $e(x, C_i) = 2$ implies $|C_i| \leq 4$.
\end{lemma}

\begin{lemma}\label{lmb} (\cite{fujita}) Suppose that $F$ is a forest with at least two components and $C$ is a triangle. Let $x_1, x_2, x_3$ be leaves of $F$ from at least two components. If $e(\{x_1, x_2, x_3\}, C) \geq 7$, then there are two disjoint cycles in $G[F \cup C]$ or there exists a triangle $C'$ in $G[F \cup C]$ such that $\omega(G[F \cup C] - C') < \omega(F)$.
\end{lemma}

\begin{lemma}\label{lmc} Suppose that $F$ is a forest with at least two components, $C$ is a triangle and $t \geq 3$ is an integer. Let $x_1, x_2, \ldots, x_t$ be leaves of $F$ from at least two components. If $e(\{x_1, x_2, \ldots , x_t\}, C) \geq 2t + 1$, then there are two disjoint cycles in $G[F \cup C]$ or there exists a triangle $C'$ in $G[F \cup C]$ such that $\omega(G[F \cup C] - C') < \omega(F)$.
\end{lemma}
\begin{proof} We prove by induction on $t$. The case $t = 3$ holds by Lemma \ref{lmb}. Suppose Lemma \ref{lmc} holds for all integers less than $t$. Now we prove the case $t$. Let $X = \{x_1, \ldots , x_t\}$. Assume there is a vertex $x_i \in X$ such that $e(x_i, C) \leq 2$. Then $X - \{x_i\}$ is a set of $t - 1$ leaves of $F$ and $e(X - \{x_i\}, C) \geq 2t - 1$. If $X - \{x_i\}$ comes from at least two components, by induction we are done. So $X - \{x_i\}$ is contained in one component $T$ of $F$.

Let $C = v_1v_2v_3v_1$. Since $e(X - \{x_i\}, C) \geq 2t - 1$, there exists a vertex $x_j \in X - \{x_i\}$ such that $e(x_j, C) \geq 3$. Because $|C| = 3$, $e(x_j, C) = 3$. Suppose there is another vertex $x_k \in X - \{x_i, x_j\}$ with $e(x_k, C) = 3$. Since $e(X - \{x_i\}, C) \geq 2t - 1$, it is easy to see that there is some $x_l \in X - \{x_i, x_j, x_k\}$ such that $e(x_l, C) \geq 1$. Because $x_j, x_k, x_l$ are three leaves from $T$, there is a path $P = x_k \cdots x_l$ in $T$ connecting $x_k$ and $x_l$ such that $x_j \notin V(P)$. Assume $x_lv_1$ is an edge, then $x_l \cdots x_k v_1 x_l$ and $x_jv_2v_3v_j$ are two disjoint cycles. Thus $e(x, C) \leq 2$ for all $x \in X - \{x_i, x_j\}$. It is not difficult to check that in this case $e(x, C) = 2$ for all $x \in X - \{x_i, x_j\}$. Now choose two vertices $x_k, x_l \in X - \{x_i, x_j\}$. Since $|C| = 3$, $x_k, x_l$ have a common neighbor, say $v_1$. Then $x_k \cdots x_lv_1x_k$ and $x_jv_2v_3x_j$ are two disjoint cycles.
\end{proof}

\begin{lemma}\label{lmd}(\cite{fujita}) Let $C$ be a cycle and $T$ be a tree with three leaves $x_1, x_2, x_3$. If $e(\{x_1, x_2, x_3\}, C) \geq 7$, then there exist two disjoint cycles in $G[C \cup T]$ or there exists a cycle $C'$ in $G[C \cup T]$ such that $|C'| < |C|$.
\end{lemma}

\begin{lemma}\label{lme} Let $C$ be a cycle and $T$ a tree with $t$ leaves $x_1, x_2, \ldots , x_t$, where $t \geq 3$ is an integer. If $e(\{x_1, x_2, \ldots , x_t\}, C) \geq 2t + 1$, then there exist two disjoint cycles in $G[C \cup T]$ or there exists a cycle $C'$ in $G[C \cup T]$ such that $|C'| < |C|$.
\end{lemma}

\begin{proof} We prove by induction on $t$.  By Lemma \ref{lmd}, the case $t = 3$ holds. Suppose the case $t - 1$ holds. Let $X = \{x_1, x_2, \ldots , x_t\}$. If $e(x_{i_0}, C) \leq 2$ for some $1 \leq i_0 \leq t$, then $e(X - \{x_{i_0}\}, C) \geq 2t - 1$, and we apply induction on $X - \{x_{i_0}\}$. Otherwise, $e(x_i, C) \geq 3$ for each $1 \leq i \leq t$, and we apply Lemma \ref{lmd} to any three vertices in $X$.
\end{proof}

\begin{lemma}\label{lmf} Let $G$ be a graph satisfying the assumption of Theorem \ref{M and Y}, and let $C_1, \ldots , C_{k - 1}$ be $k - 1$ disjoint cycles of $G$ such that $\{C_1, \ldots , C_{k - 1}\}$ is minimal. Suppose that $H = G - \cup^{k - 1}_{i = 1}C_i$ is a forest which has $t$ leaves. Then there exist $k$ disjoint cycles in $G$ or there exists a triangle $C$ in $G[H \cup C_i]$ such that $\omega(G[H \cup C_i] - C) < \omega(H)$, for some $1 \leq i \leq k - 1$.
\end{lemma}

\begin{proof} Let $X = \{x_1, x_2, \ldots , x_t\}$ be $t$ leaves of $H$ and $\mathbf{C} = \{C_1, \ldots , C_{k - 1}\}$. Clearly, $d_H(X) \leq t$. Hence, $e(X, \mathbf{C}) \geq 2kt - t + 1 - t = 2t(k - 1) + 1$. Therefore, $e(X, C_{i_0}) \geq 2t + 1$ for some $1 \leq i_0 \leq k - 1$. If the leaves in $X$ come from the same component of $H$, then using Lemma \ref{lme}, there exist $k$ disjoint cycles in $G$. So the leaves in $X$ must come from at least two components of $H$. Since $e(X, C_{i_0}) \geq 2t + 1$,  we have $e(x_i, C_{i_0}) \geq 3$ for some $1 \leq i \leq t$. By Lemma \ref{lma}, $C_{i_0}$ is a triangle. Thus using lemma \ref{lmc}, there exist $k$ disjoint cycles in $G$ or there exists a triangle  $C$ in $G[H \cup C_{i_0}]$ such that $\omega(G[H \cup C_{i_0}] - C) < \omega(H)$.
\end{proof}

Let $F$ be a forest. We call a vertex $x \in V(F)$ \emph{large degree vertex} if $d_F(x) \geq 3$.

\begin{lemma}\label{lemk}(i) Let $T$ be a tree and $L = \{x_1, \ldots , x_m\}$ a set of large degree vertices of $T$ with $d_T (x_i) = d_i$ for any $1 \leq i \leq m$. Then $T$ contains at least $\sum^m\limits_{i = 1}d_i - 2(m - 1)$ leaves.

(ii) Let $T$ be a tree and $S \subseteq V(T)$ a vertex set. If $S$ contains all the leaves of $T$, then $d_T(S) \leq 2|S| - 2$.

(iii) Let $F$ be a forest and $S \subseteq V(F)$ a vertex set. If $S$ contains all the leaves of $F$, then $d_F(S) \leq 2|S| - 2\omega(F)$.
\end{lemma}
\begin{proof}(i) Let $l$ be the number of leaves in $T$. Since $T$ is a tree, it has $|T| - 1$ edges. Clearly, $\sum^m\limits_{i = 1}d_i + l + 2(|T| - m - l) \leq 2(|T| - 1)$. Then $l \geq \sum^m\limits_{i = 1}d_i - 2(m - 1)$, done.

(ii) Suppose that $\{x_1, \ldots , x_r\}$ are all the large degree vertices that is contained in $S$ with $d_T(x_i) = d_i$ for each $1 \leq i \leq r$. By lemma \ref{lemk}-(ii), T contains at least $\sum^r\limits_{i = 1}d_i - 2(r - 1)$ leaves. Thus $S$ contains at least $\sum^r\limits_{i = 1}d_i - 2(r - 1)$ leaves. Therefore, $d_T(S) \leq \sum^k\limits_{i = 1}d_i + (\sum^r\limits_{i = 1}d_i - 2(r - 1)) + 2(|S| - r - (\sum^r\limits_{i = 1}d_i - 2(r - 1))) = 2|S| - 2$.

(iii) Let $\omega(F) = \omega$. Suppose $F = T_1 \cup \cdots \cup T_{\omega}$. Let $S_i = S \cap V(T_i)$ for any $1 \leq i \leq \omega$. Then $S = S_1 \cup \cdots \cup S_{\omega}$. By Lemma \ref{lemk}-(i), $d_F(S_i) = d_{T_i}(S_i) \leq 2|S_i| - 2$ for any $1 \leq i \leq \omega$. Hence, $d_F(S) = \sum^{\omega}\limits_{i = 1}d_F(S_i) \leq \sum^{\omega}\limits_{i = 1}(2|S_i| - 2) = 2|S| - 2\omega$.
\end{proof}

\begin{lemma}\label{lmg}(\cite{conj}) Let $C_1$ and $C_2$ be two disjoint cycles such that $|C_2| \geq 6$. Suppose that $C_2$ contains vertices with the following degree sequences from $C_2$ to $C_1$. Then $G[C_1 \cup C_2]$ contains two disjoint cycles $C'_1$ and $C'_2$ such that $|C'_1| + |C'_2| < |C_1| + |C_2|$.

$(i) \ (5, 3)$

$(ii) \ (3, 3, 1)$

$(iii) \ (3, 2, 1, 1)$

$(iv)\ (3, 1, 1, 1, 1, 1)$

$(v)\ (2, 2, 2, 2, 2)$
\end{lemma}

\begin{lemma}\label{lemma} Let $C$ be a triangle, $P$ a path with two end-vertices $x, y$. Suppose $z$ is a vertex of $G - C - P$. If $e(x, C), e(y, C) \geq 2$ and $e(z, C) = 3$, then there exist two disjoint cycles in $G[C \cup P \cup\{z\}]$.
\end{lemma}
\begin{proof} Suppose $C = v_0 v_1 v_2 v_0$. Since $e(x, C), e(y, C) \geq 2$, $x, y$ have a common neighbor on $C$, say $v_0$. Then $x \cdots  y v_0 x$ forms a cycle and $z v_1 v_2 z$ is another cycle.
\end{proof}

\begin{lemma}\label{lemmb} Let $C$ be an induced cycle with $|C| \leq 4$, $P_1 = x \cdots y$ a path with end-vertices $x, y$ and $P_2 = z \cdots w$ a path with end-vertices $z, w$. If $P_1$ and $P_2$ are disjoint , $e(x, C), e(y, C) \geq 2$ and $e(z, C) \geq 2$, $e(w, C) \geq 1$, then there exist two disjoint cycles in $G[C \cup P_1 \cup P_2]$ or a shorter cycle $C'$ than $C$ in $G[C \cup P_1 \cup P_2]$.
\end{lemma}
\begin{proof} We discuss in two cases according to the length of $C$.

First suppose that $|C| = 3$.
Let $C = v_0v_1v_2v_0$. If $z$ and $w$ share a common neighbor on $C$ say $v_0$, then $z\cdots wv_0z$ is a cycle. Since $e(x, C), e(y, C) \geq 2$, both $x$ and $y$ have at least one neighbor on  $C - \{v_0\}$. So it is easy to find another cycle. Thus $z$ and $w$ have different neighbors on $C$. Suppose $wv_0, zv_1, zv_2 \in E(G)$. Since $e(x, C), e(y, C) \geq 2$, they share a common neighbor $v_i$ on $C$. If $i = 1, 2$, then $x \cdots yv_ix$ and $z \ldots wv_0v_{3 - i}z$ are two disjoint cycles. If $i = 0$, then $x \cdots yv_0x$ and $zv_1v_2z$ are two disjoint cycles.

Second we suppose that $|C| = 4$.
Let $C = v_0v_1v_2v_3v_0$. It is obvious that $e(x, C) = 2$ and its two neighbors on $C$ are nonadjacent. Otherwise, it is easy to find a triangle. The same is true for $y$ and $z$.

If $z, w$ have a common neighbor, say $v_i$, on $C$. Then $z \cdots wv_iz$ is a cycle. Since $e(x, C)= e(y, C) = 2$, we know $e(x, C - v_i) = e(y, C - v_i) = 1$. It is not difficult to find another cycle in $G[P_1 \cup (C - v_i)]$. So $z, w$ have different neighbors on $C$. Without loss of generality, we assume $zv_1, zv_3, wv_0 \in E(G)$. Then $z \cdots wv_0v_1$ is a cycle. Since $v_0, v_1$ are adjacent on $C$, both $x, y$ have at most one neighbor in $v_0, v_1$. So $e(x, C - \{v_0, v_1\}) = e(y, C - \{v_0, v_1\}) = 1$. Again we can find another cycle in $G[P_1 \cup (C - \{v_0, v_1\})]$.
\end{proof}

\begin{lemma}\label{lemmc} Let $C$ be an induced cycle with $|C| \leq 4$, $P_1 = x \cdots y$ a path with end-vertices $x, y$ and $P_2$  connected. Suppose $P_1$, $P_2$ are disjoint and  $e(x, C), e(y, C) \geq 2$.  If there exist three vertices $u, v, w \in V(P_2)$  such that $e(u, C)$, $e(v, C), e(w, C) \geq 1$, then there exist two disjoint cycles in $G[C \cup P_1 \cup P_2]$ or a shorter cycle in $G[C \cup P_1 \cup P_2]$.
\end{lemma}
\begin{proof} We discuss in two cases according to the length of $C$.

If $|C| = 3$,
then let $C = v_0v_1v_2v_0$. Consider the three vertices $u, v, w$. If any two of them say $u, v$, share a common neighbor, say $v_0$, on $C$. Then $v_0u \cdots vv_0$ is a cycle. Since $e(x, C), e(y, C) \geq 2$ and $e(x, C - v_0), e(y, C - v_0) \geq 1$, it is easy to find another cycle in $G[P_1 \cup (C - v_0)]$. So $u, v, w$ have different neighbors on $C$. Without loss of generality, we assume $uv_0, vv_1, wv_2 \in E(G)$. Since $e(x, C), e(y, C) \geq 2$, they share at least one common neighbor say $v_0$, on $C$. Then $v_0x\cdots yv_0$ and $v\cdots wv_2v_1v$ are two disjoint cycles.

So $|C| = 4$.
Let $C = v_0v_1v_2v_3v_0$. Clearly, $e(x, C), e(y, C) = 2$ and both $x, y$ have nonadjacent vertices on $C$. Otherwise, we can find a triangle.

If any two of $u, v, w$ say $u, v$ share a common neighbor say $v_i$, on $C$. Then $v_iu \cdots v v_i$ is a cycle. As $e(x, C), e(y, C) = 2$ and $e(x, C - v_i), e(y, C - v_i) = 1$, it is easy to find another cycle in $G[P_1 \cup (C - v_i)]$.

So $u, v, w$ have different neighbors on $C$. Without loss of generality, we assume $uv_0, vv_1, wv_2 \in E(G)$. Consider $x, y$. If $x, y$ have common neighbors, say $v_0, v_2$, on $C$, then $v_0x \cdots yv_0$ and $v_1v \cdots w v_2v_1$ are two disjoint cycles. So $x, y$ have different neighbors on $C$. Suppose $xv_0, xv_2, yv_1, yv_3 \in E(G)$. Then $x\cdots yv_3v_0x$ and $v_1v \cdots w v_2v_1$ are two disjoint cycles.
\end{proof}

\section{Proof of Theorem \ref{M and Y}}

Let $G$ be an edge-maximal counterexample which satisfies the condition of Theorem \ref{M and Y}. Since a complete graph of order at least $(2t - 1)k$ contains $k$ disjoint cycles, $G$ is not complete. Let $x$ and $y$ be two non-adjacent vertices of $G$. Then $G' = G + xy$ is not a counterexample by the maximality of $G$. Hence $G'$ contains $k$ disjoint cycles $C_1, \ldots , C_k$ and without loss of generality, we may assume that $xy \in E(C_k)$.  This means that $G$ contains $k - 1$ disjoint cycles $C_1, \ldots , C_{k - 1}$. Let $\mathbf{C} = \{C_1, \ldots , C_{k - 1}\}$ and $H = G - \mathbf{C}$. Choose $C_1, \ldots , C_{k - 1}$ such that
\begin{equation}\label{insert}
\mathbf{C} \mbox{ is minimal.}
\end{equation}

Subject to (\ref{insert}),

\begin{equation}\label{first}
\omega(H) \mbox{ is minimum.}
\end{equation}

Clearly, any cycle $C \in \mathbf{C}$ has no chord by the minimality and $H$ is a forest otherwise $G$ would contain $k$ disjoint cycles.

We distinguish two cases according to the value of $|H|$.
\\
\\
{\bf CASE 1} \ $|H| \geq 3t - 1$.
\\
\\
Suppose that $\omega(H) = \omega$ and $H = T_1 \cup \cdots \cup T_{\omega}$, where $T_i$ is a tree for each $1 \leq i \leq \omega$. Clearly, by Lemma \ref{lmf}, $1 \leq \omega \leq t - 1$. Since a tree is a bipartite graph, there exists a vertex partition $(V_{i1}, V_{i2})$ of $V(T_i)$ such that $V(T_i) = V_{i1} \cup V_{i2}$ and $V_{i1}, V_{i2}$ are two disjoint independent sets of $T_i$. Let $X = \cup^{\omega}_{i = 1}V_{i1}$ and $Y = \cup^{\omega}_{i = 1}V_{i2}$. Then $X$ and $Y$ are two disjoint independent sets of $H$ and $V(H) = X \cup Y$. Without loss of generality, we may assume $|X| \geq |Y|$.

\begin{claim}\label{claimb} There exist two disjoint independent sets in $H$ such that each of which contains $t$ vertices.
\end{claim}
\begin{proof} Since $|X| + |Y| = |H| \geq 3t - 1$ and $|X| \geq |Y|$,  we see that $|X| > t$. If $t < |X| < 2t$, then $|Y| \geq t$. There exists an independent set of size $t$ in both $X$ and $Y$.  If $|X| \geq 2t$, then we can find two disjoint independent sets of size $t$ in $X$.
\end{proof}

Let $X_1, X_2$ be those disjoint independent sets of Claim \ref{claimb}. Denote their union by $I$. Choose $I$ such that it contains as many as leaves of $H$. We claim that $I$ contains all the leaves of $H$.

\begin{claim}\label{call} $I$ contains all the leaves of $H$.
\end{claim}
\begin{proof} Suppose there exists a leaf $x$ such that $x \notin I$. If $x$ is an isolated vertex or its neighbor $z \notin I$, replace a vertex $y \in I$ where $d_H(y) \geq 2$ by $x$, then we get a $I'$ which contains more leaves than $I$. By Lemma \ref{lmf}, this kind of vertex $y$ does exist. Therefore, $z \in I$. If $z$ is a leaf and without loss of generality assume $z \in X_1$, then add $x$ to $X_2$ by replacing a vertex $y$ with $d_H(y) \geq 2$. If $z$ is not a leaf, replace $z$ by $x$. In either case, we get a $I'$ which contains more leaves than $I$.
\end{proof}

Using Lemma \ref{lemk}-(iii), $d_H(I) \leq 2|I| - 2\omega = 4t - 2\omega$. Thus $e(I, {\mathbf{C}}) \geq 2(2kt - t + 1) - (4t - 2\omega) = 4t(k - 1) + 2\omega + 2 - 2t$. By pigeon hole principle, there exists some $C \in \mathbf{C}$ such that
\begin{eqnarray}\label{eq1}
e(I, C) &\geq& 4t + \frac{2\omega + 2 - 2t}{k - 1} \nonumber \\
       &\geq& 2t + 2\omega + 2, \mbox{  since $k \geq 2$ and $\omega \leq t - 1$}.
\end{eqnarray}

Let $Y$ be the set of all vertices $y \in I$ such that $e(y, C) \geq 2$. By Lemma \ref{lma}, $e(y, C) \leq 3$ for any $y \in Y$. Thus
\begin{equation}\label{eq2}
e(I, C) \leq 3|Y| + (2t - |Y|) = 2|Y| + 2t.
\end{equation}

By (\ref{eq1}) and (\ref{eq2}), we get that $2|Y| + 2t \geq 2t + 2\omega + 2$, i.e.
\begin{equation}\label{eq3}
|Y| \geq \omega + 1.
\end{equation}

Therefore, there exists some $T_i \subseteq H$ such that $|Y \cap V(T_i)| \geq 2$. Since $Y \neq \emptyset$, using Lemma \ref{lma}, we see $|C| \leq 4$.

\begin{claim}\label{claimd} For any $T_j$ with $j \neq i$, $e(T_j, C) \leq 2$.
\end{claim}
\begin{proof} Suppose otherwise $e(T_j, C) \geq 3$. First, if there exist some $u \in V(T_j)$ such that $e(u, C) \geq 3$, using Lemma \ref{lma}, we know $|C| = 3$. Then using Lemma \ref{lemma}, we get two disjoint cycles in $G[C \cup T_i \cup T_j]$. Second, if there are two vertices $u, v \in V(T_j)$ such that $e(u, C) \geq 2$ and $e(v, C) \geq 1$, by Lemma \ref{lemmb} and (\ref{insert}), we get two disjoint cycles in $G[C \cup T_i \cup T_j]$. Finally, there are three vertices $u, v, w \in V(T_j)$ such that $e(u, C), e(v, C), e(w, C) \geq 1$, by Lemma \ref{lemmc} and (\ref{insert}), again we get two disjoint cycles in $G[C \cup T_i \cup T_j]$.
\end{proof}

From (\ref{eq1}) and Claim \ref{claimd},
\begin{equation}\label{eq4}
e(I \cap V(T_i), C) \geq 2t + 2\omega + 2 - 2(\omega - 1) \geq 2t + 4.
\end{equation}

By Claim \ref{claimd}, $|Y \cap V(T_j)| \leq 1$ for any $j \neq i$. Hence, by (\ref{eq3}), $|Y \cap V(T_i)| \geq |Y| - (\omega - 1) \geq (\omega + 1) - (\omega - 1) = 2$. We claim that at most one vertex $x \in V(T_i) \cap I$ with $e(x, C) \geq 3$. Suppose there are two vertices $x, y \in V(T_i) \cap I$ with $e(x, C), e(y, C) \geq 3$. Then by Lemma \ref{lma}, $|C| = 3$. By (\ref{eq4}), there must a vertex $z \in  V(T_i) \cap I$ with $e(z, C) \geq 1$. So there is a path in $T_i$ connecting, say $x$, to $z$ with $y$ not on it. Without loss of generality, let $C = v_0v_1v_2v_0$ and $zv_0 \in E(G)$. Now $x \cdots z v_0x$ and $yv_1v_2y$ are two disjoint cycles in $G[C \cup T_i]$, a contradiction.

Actually, $|Y \cap V(T_i)| \geq 3$. Assume $|Y \cap V(T_i)| \leq 2$. Let $U = \{v \in I \cap V(T_i) : e(v, C) = 1\}$. Then by (\ref{eq4}), $|U| \geq 2t + 4 - (3 + 2) = 2t - 1$. Hence, $|I| \geq |I \cap V(T_i)| \geq 2t - 1 + 2 = 2t + 1$, a contradiction.

Let $A = \{v \in I \cap V(T_i) : e(v, C) \geq 2\}$ and $B = \{v \in I \cap V(T_i) : e(v, C) = 1 \}$. By Lemmas \ref{lemma} and \ref{lemmb}, there exists a center vertex $u^\ast$ such that for any $x, y \in A$ the paths from $x$ to $u^\ast$ and from $y$ to $u^\ast$ are disjoint except the end-vertex $u^\ast$. Moreover, for any $x \in A$, the path from $x$ to $u^\ast$ contains no vertex in $A \cup B$.  For any two vertices $u, v \in B$, let $P_1$ be the path joining $u$ to the center vertex $u^\ast$ and $P_2$ be the path joining $v$ to the center vertex $u^\ast$. By Lemma \ref{lemmb}, either $P_1$ and $P_2$ are disjoint except for the center vertex $u^\ast$ or $P_1 \subseteq P_2$ or $P_2 \subseteq P_1$. By Lemma \ref{lemmc}, at most two vertices $u, v \in B$ satisfying the latter case. That is to say, for any $u \in B$, the path from $u$ to $u^\ast$ contains at most one other vertex in $B$. By Lemma \ref{lemmb}, for any $u \in B$, the path from $u$ to $u^\ast$ contains no  vertex in $A$. From the analysis above, we can see that the structure of $T_i$ is something like an extension of a star, see Fig.1.

\begin{figure}[h]
  \centering
  \includegraphics[width=6cm]{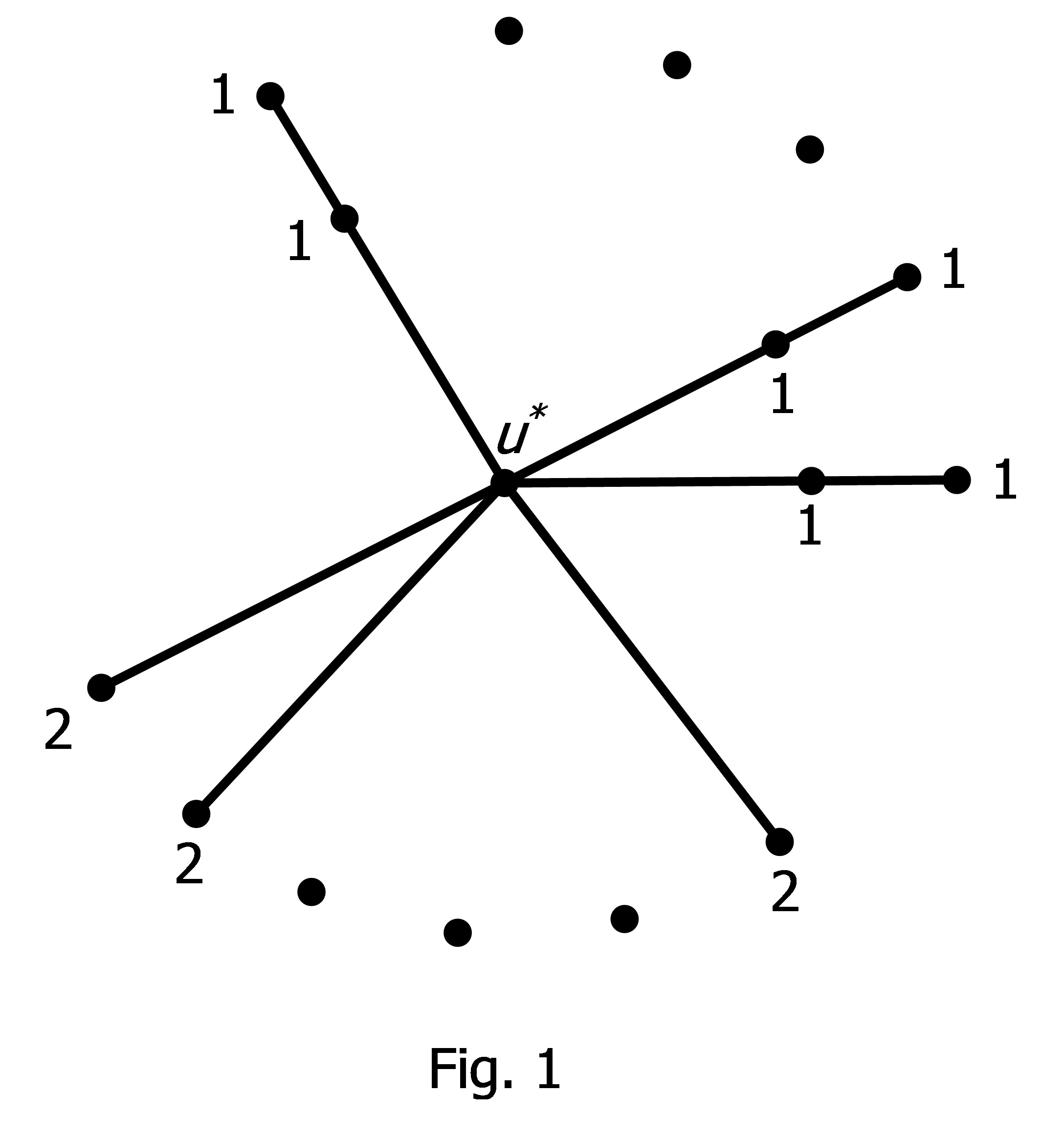}\\

\end{figure}

Using Lemma \ref{lmf}, we know $T_i$ contains at most $(t - 1) - (w - 1) = t - w$ leaves. Hence, if we let $|A| = a$ and $|B| = b$, then $b\leq 2(t - \omega - (a - 1))$, $a - 1$ by considering the center vertex may belong to $A$. Therefore, $e(I \cap V(T_i), C) \leq 3 + 2a + b\leq 2t - 2\omega + 5$. By (\ref{eq4}), $2t - 2\omega + 5 \geq 2t + 4$, i.e. $1 - 2\omega \geq 0$, a contradiction.
\\
\\
{\bf CASE 2} \  $|H| \leq 3t - 2$.
\\
\\
{\bf Subcase 1} $k = 2$.
\\
\\
In this case, $\mathbf{C}$ has only one cycle $C$, and $H = G - C$.  Suppose $|H| \geq 2t - 1$. Since $H$ is a forest, it is  bipartite. There is a partition of $V(H) = V_1 \cup V_2$ such that $V_i$ is an independent set for $i = 1, 2$.  Since $|H| \geq 2t - 1$, one of $V_1, V_2$ has at least $t$ vertices. Thus there exists an independent set $X \subseteq H$ with $|X| = t$. Choose $X$ such that it contains as many leaves of $H$ as possible.

\begin{claim}\label{clmf} $d_H(X) \leq 2t - 2$.
\end{claim}
\begin{proof} Let $X_i = V(T_i) \cap X$, then $X = \cup_{1 \leq i \leq \omega} X_i$. If $T_i \notin \{K_1, K_2\}$, then we claim that $X_i$ contains all the leaves of $T_i$. Actually, suppose $x$ is a leaf of $T_i$ and $x \notin X_i$. Consider the neighbor of $x$ in $T_i$, and denote it by $y$, obviously $d_{T_i}(y) \geq 2$. If $y \in X_i$, then replace $X_i$ by $X'_i = (X - \{y\}) \cup \{x\}$ and $X$ by $X' = (X - X_i) \cup X'_i$, we get an independent set which contains more leaves than $X$, a contradiction. So $y \notin X_i$. By Lemma \ref{lmf}, $H$ contains at most $t - 1$ leaves, thus there exists a vertex $z \in X$ with $d_H(z) \geq 2$. Replace $X$ by $X' = (X - \{z\}) \cup \{x\}$, again we get an independent set which contains more leaves than $X$. Therefore, by Lemma \ref{lemk}-(ii), $d_H(X_i) \leq 2|X_i| - 2$.

For those $T_i \in \{K_1, K_2\}$, it is easy to see that $d_H(X_i) \leq |X_i| = 1$. Therefore, $d_H(X) \leq \sum_{T_i \notin\{K_1, K_2\}}(2|X_i| - 2) + \sum_{T_i \in\{K_1, K_2\}}|X_i|$ $\leq 2|X| - 2$.
\end{proof}

Thus $e(X, C) \geq 2kt - t + 1 - (2t - 2) = t + 3$. Therefore, $e(x, C) \geq 2$ for some $x \in I$.  By Lemma \ref{lma}, this means $|C| \leq 4$. Thus, $|H| \geq |G| - |C| \geq (4t - 2) - 4 = 4t - 6 \geq 3t - 1$, since $t \geq 5$. By Case 1, this is a contradiction.

Therefore, $|H| \leq 2t - 2$. Then $|C| \geq |G| - |H| \geq (4t - 2) - (2t - 2) = 2t + 4$. Thus there exist two disjoint independent sets $X_1, X_2$ in $C$ such that each has $t$ vertices. Denote their union by $I$. Since $C$ has no chord, $d_C(I) = 4t$. Then $e(I, H) \geq 2(2kt - t + 1) - 4t = 4kt - 6t + 2 = 2t + 2$, since $k = 2$. On the other hand, since $|C| \geq 2t + 4 > 4$, by Lemma \ref{lma}, $e(v, C) \leq 1$ for any $v \in H$. Hence, $e(H, C) \leq |H|$. Therefore, $2t + 2 \leq e(I, H) \leq e(C, H) \leq |H|$, i.e. $|H| \geq 2t + 2$, a contradiction.
\\
\\
{\bf Subcase 2} $k \geq 3$.
\\
\\
Let $C \in \mathbf{C}$ be the longest cycle. Suppose $|C| \geq 2t$. Denote by $|C| = st + r$. Thus there exist $s$ independent sets $X_1, \ldots , X_s$ in $C$, where $s \geq 2$ and $0 \leq r \leq t - 1$. Let $I = X_1 \cup \cdots \cup X_s$. Since $C$ has no chord, $d_C(I) = 2st$. Moreover, since $|C| \geq 5$, by Lemma \ref{lma}, $e(I, H) \leq |H| \leq 3t - 2$. Hence,
\begin{eqnarray}\label{eqa}
e(I, \mathbf{C} - C) &\geq& s(2kt - t + 1) - 2st - (3t - 2)  \nonumber \\
                       &=&    2st(k - 2) + s(t + 1) - (3t - 2). \nonumber
\end{eqnarray}
Therefore, there exists some $C' \in \mathbf{C} - C$ such that

\begin{equation}\label{eqb}
e(I, C') \geq 2st + \frac{s(t + 1) - (3t - 2)}{k - 2}
\end{equation}
We now discuss in two cases according to $s$.
\\
\\
Let $h = max \{e(v, C') : v \in I\}$ and $Z = \{v \in I : N_{C'}(v) \neq \emptyset\}$.

{\bf Subcase 2.1} $s \geq 3$.
\\
\\
In this case, by (\ref{eqb}), $e(I, C') \geq 2st + 1$. Thus $e(x, C') \geq 3$ for some $x \in I$. Then $3 \leq h \leq |C'| \leq |C| = st + r$. So $e(I - x, C') \geq 2st + 1 - (st + r) = st - r + 1 \geq 2t + 2 \geq 12$, since $s \geq 3$, $r \leq t - 1$ and $t \geq 5$. This implies that $N_{C'}(C - x) \neq \emptyset$. Then $|Z| \geq 2$.

Suppose that $|Z| = 2$. Then $d_{C'}(v) \geq 12$ for any $v \in Z$. By Lemma \ref{lmg}-(i), $G[C \cup C']$ contains two shorter disjoint cycles, contradicts (\ref{insert}). Suppose that $|Z| \geq 6$, by Lemma \ref{lmg}-(iv), $G[C \cup C']$ contains two shorter disjoint cycles, again a contradiction. Therefore, $3 \leq |Z| \leq 5$. Since $e(I - x, C') \geq 12$, $d_{C'}(y) \geq 3$ for some $y \in I - x$. By Lemma \ref{lmg}-(ii), we get two shorter cycles in $G[C \cup C']$. In any case, we get a contradiction.
\\
\\
{\bf Subcase 2.2} $s = 2$.
\\
\\
By (\ref{eqb}) and $k \geq 3$, $e(I, C') \geq 4t + \frac{4 - t}{k - 2} \geq 3t + 4$. Then $2 \leq h \leq |C'| \leq 2t + r$. Without loss of generality, assume $x \in I$ satisfies $e(x, C') = h$. So $e(I - x, C') \geq 3t + 4 - (2t + r) = t - r + 4 \geq 5$. That is to say, $N_{C'}(I - x) \neq \emptyset$. Then $|Z| \geq 2$.

Suppose $h = 2$. Let $Y = \{v \in I : e(v, C') = 2\}$. Then $e(I, C') \leq 2|Y| + (|I| - |Y|) = |Y| + 2t$. Since $e(I, C') \geq 3t + 4$, we get $|Y| + 2t \geq 3t + 4$, i.e. $|Y| \geq t + 4 > 5$. By Lemma \ref{lmg}-(v), we get two shorter cycles in $G[C \cup C']$. Therefore, $h \geq 3$.

If $|Z| = 2$, then $e(v, C') \geq 5$ for any $v \in Z$. By Lemma \ref{lmg}-(i), there are two shorter cycles in $G[C \cup C']$. If $|Z| \geq 6$, by Lemma \ref{lmg}-(iv), $G[C \cup C']$ contains two shorter disjoint cycles,a contradiction. So $5 \geq |Z|\geq 3$. Since $e(I - x, C') \geq 5$, it is not difficult to check that we can get one of the following degree sequence $S$ from $C$ to $C'$: ($h, 4, 1$), ($h$, 3, 2), ($h$, 3 , 1, 1), ($h$, 2, 2, 1) and ($h$, 2, 1, 1, 1). Using Lemma \ref{lmg}, in any case we get two shorter cycles in $G[C \cup C']$.

Therefore, $|C| \leq 2t - 1$ for any $C \in \mathbf{C}$. So $|\mathbf{C}| \leq (k - 1)(2t - 1)$. Hence $|H| \geq |G| - |\mathbf{C}| \geq (2t - 1)k - (2t - 1)(k - 1) = 2t - 1$. As discussed before, $H$ contains an independent set $X$ with $t$ vertices. Choose $X$ such that it has as many leaves as possible. Using Claim \ref{clmf}, we know that $d_H(X) \leq 2t - 2$. Thus, $e(X, \mathbf{C}) \geq (2kt - t + 1) - (2t - 2) = 2t(k - 1) + 3 - t$. Therefore, there exists some $C \in \mathbf{C}$ such that
\begin{eqnarray}\label{eq5}
e(X, C) &\geq& 2t + \frac{3 - t}{k - 1} \nonumber \\
&>& t + 3, \mbox{  since }  k > 2  \mbox{ and } t \geq 5. \nonumber
\end{eqnarray}
This means $e(x, C) \geq 2$ for some $x \in X$. By Lemma \ref{lma}, we see $|C| \leq 4$. Hence, $|\mathbf{C}| \leq 4 + (k - 2)(2t - 1)$. It follows that $|H| \geq |G| - |\mathbf{C}| \geq (2t - 1)k - (4 + (2t - 1)(k - 2) = 4t - 6 \geq 3t - 1$, since $t \geq 5$. By Case 1, we get a contradiction. We finish our proof of Theorem \ref{M and Y}.
\\
\\


\begin{thebibliography}{99}

\bibitem{candh}
K. Corr$\acute{a}$di, A. Hajnal, On the maximal number of independent circuits in a graph, Acta Math. Acad. Sci. Hung. 14 (1963) 423-439.

\bibitem{Dirac}
G. A. Dirac, On the maximal number of independent triangles in graphs, Abh. Math. Semin. Univ. Hamb. 26 (1963) 78-82.

\bibitem{enomoto}
H. Enomoto, On the existence of disjoint cycles in a graph, Combinatorica 18 (4) (1998) 487-492.

\bibitem{Erdos}
P. Erd$\ddot{o}$s, L. P$\acute{o}$sa, On the maximal number of disjoint circuits of a graph, Pul. Math. Debrecen. 9 (1962) 3-12.

\bibitem{fujita}
S. Fujita, H. Matsumura, M. Tsugaki, T. Yamashita, Degree sum conditions and vertex disjoint cycles in a graph, Aus. J Combin. 35 (2006) 237-251.

\bibitem{conj}
R. J. Gould, K. Hirohata, A. Keller, On vertex disjoint cycles and degree sum conditions, to appear.

\bibitem{Jiao}
Z. Jiao, H. Wang, J. Yan, Disjoint cycles in graphs with distance degree sum conditions, Discrete Math. 340 (2017) 1203-1209.

\bibitem{just}
P. Justesen, On independent circuits in finite graph and a conjecture of Erd\H{o}s and P$\acute{o}$sa, Ann. Disc. Math. 41 (1989) 299-306.

\bibitem{wang2}
H. Wang, Disjoint long cycles in a graph, Science in China Ser. A: Math. 56 (2013) 1983-1998.

\bibitem{wang1}
H. Wang, On the maximum number of independent cycles in a bipartite graph, J Combin. Theory, Ser. B 67 (1996) 152-164.

\bibitem{wang}
H. Wang, On the maximum number of independent cycles in a graph, Discrete Math. 205 (1) (1999) 183-190.

\bibitem{Yan}
J. Yan, Y. Gao, On Enomoto's problems in a bipartite graph, Science in China Ser. A: Math. 52 (9) (2009) 1947-1954.
\end{thebibliography}
\end{document}